\documentclass{amsart}%
\usepackage{amsfonts}
\usepackage{amsmath}
\usepackage{amssymb}
\usepackage{graphicx}%
\providecommand{\U}[1]{\protect\rule{.1in}{.1in}}
\newtheorem{theorem}{Theorem}
\theoremstyle{plain}

\newtheorem{corollary}{Corollary}

\newtheorem{definition}{Definition}
\newtheorem{example}{Example}

\newtheorem{lemma}{Lemma}

\newtheorem{proposition}{Proposition}
\newtheorem{remark}{Remark}

\numberwithin{equation}{section}
\begin{document}
\title[Quasi J-ideals of Commutative Rings ]{Quasi J-ideals of Commutative Rings }
\author{Hani A. Khashan }
\address{Department of Mathematics, Faculty of Science, Al al-Bayt University, Al
Mafraq, Jordan.}
\email{hakhashan@aabu.edu.jo}
\author{Ece YETKIN\ CELIKEL}
\address{Department of Electrical-Electronics Engineering, Faculty of Engineering,
Hasan Kalyoncu University, Gaziantep, Turkey.}
\email{ece.celikel@hku.edu.tr, yetkinece@gmail.com}
\thanks{This paper is in final form and no version of it will be submitted for
publication elsewhere.}
\date{December, 2020}

\begin{abstract}
Let $R$ be a commutative ring with identity. In this paper, we introduce the
concept of quasi $J$-ideal which is a generalization of $J$-ideal. A proper
ideal of $R$ is called a quasi $J$-ideal if its radical is a $J$-ideal. Many
characterizations of quasi $J$-ideals in some special rings are obtained. We
characterize rings in which every proper ideal is quasi $J$-ideal. Further, as
a generalization of presimplifiable rings, we define the notion of quasi
presimplifiable rings. We call a ring $R$ a quasi presimplifiable ring if
whenever $a,b\in R$ and $a=ab$, then either $a$ is a nilpotent or $b$ is a
unit. It is shown that a proper ideal $I$ that is contained in the Jacobson
radical is a quasi $J$-ideal (resp. $J$-ideal) if and only if $R/I$ is a quasi
presimplifiable (resp. presimplifiable) ring.

\end{abstract}
\maketitle

\section{Introduction}

Throughout this paper, we shall assume unless otherwise stated, that all rings
are commutative with non-zero identity. We denote the nilradical of a ring
$R$, the Jacobson radical of $R$, the set of unit elements of $R$, the set of
zero-divisors and the set of all elements that are not quasi-regular in $R$ by
$N(R),$ $J(R),$ $U(R)$, $Z(R),$ and $NZ(R),$ respectively. In \cite{Tek}, the
concept of $n$-ideals in commutative rings is defined and studied. A proper
ideal $I$ of $R$ is said to be a $n$-ideal if whenever $a,b\in R$ with $ab\in
I$ and $a\notin N(R)$, then $b\in I$. Recently, as a generalization of
$n$-ideals, the notion of $J$-ideals is introduced and investigated in
\cite{Hani}. A proper ideal $I$ of $R$ is called a $J$-ideal if whenever
$a,b\in R$ with $ab\in I$ and $a\notin J(R)$, then $b\in I$.

The aim of this article is to extend the notion of $J$-ideals to quasi
$J$-ideals. For the sake of thoroughness, we give some definitions which we
will need throughout this study. For a proper ideal $I$ a ring $R$, let
$\sqrt{I}=\{r\in R$ : there exists $n\in%
\mathbb{N}
$ with $r^{n}\in I\}$ denotes the radical of $I$ and $(I:x)$ denotes the ideal
$\{r\in R$ : $rx\in I\}$. Let $M$ be a unitary $R$-module. Recall that the
idealization $R(+)M=\{(r,m):r\in R,$ $m\in M$\} is a commutative ring with the
addition $(r_{1},m_{1})+(r_{2},m_{2})=(r_{1}+r_{2},m_{1}+m_{2})$ and
multiplication $(r_{1},m_{1})(r_{2},m_{2})=(r_{1}r_{2},r_{1}m_{2}+r_{2}m_{1}%
)$. For an ideal $I$ of $R$ and a submodule $N$ of $M$, it is well-known that
$I(+)N$ is an ideal of $R(+)M$ if and only if $IM\subseteq N$ \cite[Theorem
3.1]{A}. We recall also from \cite[Theorem 3.2]{A} that $\sqrt{I(+)N}=\sqrt
{I}(+)M$, and the Jacobson radical of $R(+)M$ is $J(R(+)M)=J(R)(+)M$. For the
other notations and terminologies that are used in this article, the reader is
referred to \cite{Attiyah}.

We summarize the content of this article as follows. In Section 2, we study
the basic properties of quasi $J$-ideals of a ring $R$. Among many results in
this section, we first start with an example of a quasi $J$-ideal that is not
a $J$-ideal. In Theorem \ref{eq}, we give a characterization for quasi
$J$-ideals. In Theorem \ref{quasi}, we conclude some equivalent conditions
that characterize quasi-local rings. The relations among primary, $\delta_{1}%
$-$n$-ideal and quasi $J$-ideals are clarified (Proposition \ref{delta}).
Moreover, Example \ref{edelta} and Example \ref{eC} are presented showing that
the converses of the used implications are not true in general. Further, in
Theorem \ref{max}, we show that every maximal quasi $J$-ideal is a $J$-ideal.
In Theorem \ref{zero}, we characterize quasi $J$-ideals\ of zero-dimensional
rings in terms of quasi primary ideals. Moreover, the behavior of quasi
$J$-ideals\textbf{ }in polynomial rings, power series rings, localizations,
direct product of rings, idealization rings are investigated (Proposition
\ref{pol}, Proposition \ref{S}, and Proposition \ref{pp}, Remark \ref{r} and
Proposition \ref{pide}).

In Section 3, we introduce quasi presimplifiable rings as a new generalization
of\ presimplifiable rings. We call a ring $R$ quasi presimplifiable if
whenever $a,b\in R$ with $a=ab$, then $a\in N(R)$ or $b\in U(R)$. Clearly, the
classes of presimplifiable and quasi presimplifiable reduced rings coincide.
However, in Example \ref{exp}, we show that in general this generalization is
proper. In Proposition \ref{q1}, it is shown that a ring $R$ is quasi
presimplifiable if and only if $NZ(R)\subseteq J(R)$. The main objective of
the section is to characterize a $J$-ideal (resp. a quasi $J$-ideal) of $R$ as
the ideal $I$ for which $R/I$ is a presimplifiable (resp. quasi
presimplifiable) ring. This characterization is used to justify more results
concerning the class of $J$-ideals (resp. quasi $J$-ideals). For example, in
Theorem \ref{dun}, it is shown that if $\{I_{\alpha}:\alpha\in\Lambda\}$ is a
family of $J$-ideals (resp. quasi $J$-ideals) over a system of rings
$\{R_{\alpha}:\alpha\in\Lambda\}$, then $I=\bigcup\limits_{\alpha\in\Lambda
}\varphi_{\alpha}(I_{\alpha})$ is a $J$-ideal (resp. quasi $J$-ideal) of
$R=\underrightarrow{\lim}R_{\alpha}$.

\section{Properties of Quasi $J$-ideals}

\begin{definition}
Let $R$ be a ring. A proper ideal $I$ of $R$ is said to be a quasi $J$-ideal
if $\sqrt{I}$ is a $J$-ideal.
\end{definition}

It is clear that every $J$-ideal is a quasi $J$-ideal. However, this
generalization is proper and the following is an example of a quasi $J$-ideal
in a certain ring which is not a $J$-ideal.

\begin{example}
Consider the idealization ring $R=%
\mathbb{Z}
(+)%
\mathbb{Z}
$. Then $I=0(+)%
\mathbb{Z}
$ is a $J$-ideal of $R$ since $0$ is a $J$-ideal of $%
\mathbb{Z}
$ by \cite[Proposition 3.12]{Hani}. Now, $\sqrt{0(+)2%
\mathbb{Z}
}=\sqrt{0}(+)%
\mathbb{Z}
=0(+)%
\mathbb{Z}
$ is a $J$-ideal of $R$, and thus $0(+)2%
\mathbb{Z}
$ is a quasi $J$-ideal of $R$. However, $0(+)2%
\mathbb{Z}
$ is not a $J$-ideal of $R$ since for example $(0,1),(2,0)\in R$ with
$(2,0)\cdot(0,1)=(0,2)\in0(+)2%
\mathbb{Z}
$ and $(2,0)\notin J(R)=J(%
\mathbb{Z}
)(+)%
\mathbb{Z}
=0(+)%
\mathbb{Z}
$ but $(0,1)\notin0(+)2%
\mathbb{Z}
$.
\end{example}

Our starting point is the following characterization for quasi $J$-ideals.

\begin{theorem}
\label{eq}Let $I$ be a proper ideal of a ring $R.$ Then the following
statements are equivalent:
\end{theorem}

\begin{enumerate}
\item $I$ is a quasi $J$-ideal of $R.$

\item If $a\in R$ and $K$ is an ideal of $R$ with $aK\subseteq I$, then $a\in
J(R)$ or $K\subseteq\sqrt{I}.$

\item If $K$ and $L$ are ideals of $R$ with $KL\subseteq I$, then $K\subseteq
J(R)$ or $L\subseteq\sqrt{I}.$

\item If $a,b\in R$ and $ab\in I$, then $a\in J(R)$ or $b\in\sqrt{I}.$
\end{enumerate}

\begin{proof}
(1)$\Rightarrow$(2) Suppose that $I$ is a quasi $J$-ideal of $R,$ $aK\subseteq
I$ and $a\notin J(R).$ Since $\sqrt{I}$ is a $J$-ideal, $\sqrt{I}=(\sqrt
{I}:a)$ by \cite[Proposition 2.10]{Hani}. Thus $K\subseteq(I:a)\subseteq
(\sqrt{I}:a)=\sqrt{I}.$

(2)$\Rightarrow$(3) Suppose that $KL\subseteq I$ and $K\nsubseteq J(R).$ Then
there exists $a\in K\backslash J(R).$ Since $aL\subseteq I$ and $a\notin
J(R)$, we have $L\subseteq\sqrt{I}$ by our assumption.

(3)$\Rightarrow$(4) Suppose that $a,b\in R$ and $ab\in I$. The result follows
by letting $K=<a>$ and $L=<b>$ in (3).

(4)$\Rightarrow$(1) We show that $\sqrt{I}$ is a $J$-ideal. Suppose that
$ab\in\sqrt{I}$ and $a\notin J(R)$. Then there exists a positive integer $n$
such that $a^{n}b^{n}\in I$ and $a\notin J(R)$. It follows clearly that
$a^{n}\notin J(R)$ and so $b^{n}\in\sqrt{I}$ by (4). Therefore, $b\in
\sqrt{\sqrt{I}}=\sqrt{I}$ and $I$ is a quasi $J$-ideal.
\end{proof}

As a consequence of Theorem \ref{eq}, we have the following.

\begin{corollary}
Let $L$ be an ideal of a ring $R$ such that $L\nsubseteq J(R)$. Then
\end{corollary}

\begin{enumerate}
\item If $I$ and $K$ are quasi $J$-ideals of $R$ with $IL=KL$, then $\sqrt
{I}=\sqrt{K}.$

\item If for an ideal $I$ of $R$, $IL$ is a quasi $J$-ideal, then $\sqrt
{IL}=\sqrt{I}.$
\end{enumerate}

Let $I$ be a proper ideal of $R$. We denote by $J(I)$, the intersection of all
maximal ideals of $R$ containing $I$. Next, we obtain the following
characterization for quasi $J$-ideals of $R.$

\begin{proposition}
\label{J(I)}Let $I$ be an ideal of $R.$ Then the following statements are equivalent:
\end{proposition}

\begin{enumerate}
\item $I$ is a quasi $J$-ideal of $R.$

\item $I\subseteq J(R)$ and if whenever $a,b\in R$ with $ab\in I$, then $a\in
J(I)$ or $b\in\sqrt{I}.$
\end{enumerate}

\begin{proof}
(1)$\Rightarrow$(2) Suppose $I$ is a quasi $J$-ideal of $R.$ Since $\sqrt{I}$
is a $J$-ideal, then $I\subseteq\sqrt{I}\subseteq J(R)$ by \cite[Proposition
2.2]{Hani}. Now, (2) follows clearly since $J(R)\subseteq J(I)$.

(2)$\Rightarrow$(1) Suppose that $ab\in I$ and $a\notin J(R).$ Since
$I\subseteq J(R),$ we conclude that $J(I)\subseteq J(J(R))=J(R)$ and so we get
$a\notin J(I).$ Thus, $b\in\sqrt{I}$ and $I$ is a quasi $J$-ideal of $R.$
\end{proof}

In the following theorem, we characterize rings in which every proper
(principal) ideal is a quasi $J$-ideal.

\begin{theorem}
\label{quasi}For a ring $R$, the following statements are equivalent:
\end{theorem}

\begin{enumerate}
\item $R$ is a quasi-local ring.

\item Every proper principal ideal of $R$ is a $J$-ideal.

\item Every proper ideal of $R$ is a $J$-ideal.

\item Every proper ideal of $R$ is a quasi $J$-ideal.

\item Every proper principal ideal of $R$ is a quasi $J$-ideal.

\item Every maximal ideal of $R$ is a quasi $J$-ideal.
\end{enumerate}

\begin{proof}
(1)$\Rightarrow$(2)$\Rightarrow$(3) is clear by \cite[Proposition 2.3]{Hani}.

Since (3)$\Rightarrow$(4)$\Rightarrow$(5) is also clear, we only need to prove
(5)$\Rightarrow$(6) and (6)$\Rightarrow$(1).

(5)$\Rightarrow$(6) Assume that every proper principal ideal of $R$ is a quasi
$J$-ideal. Let $M$ be a maximal ideal of $R.$ Suppose that $ab\in M$ and
$a\notin\sqrt{M}=M.$ Since $<ab>$ is proper in $R$, $(ab)$ is a quasi
$J$-ideal by our assumption. Since $ab\in<ab>$ and clearly $a\notin\sqrt
{<ab>}$, we conclude that $b\in J(R),$ as required.

(6)$\Rightarrow$(1) Let $M$ be a maximal ideal of $R$. Then $M$ is a quasi
$J$-ideal by (6) which implies $M=\sqrt{M}\subseteq J(R)$ by \cite[Proposition
2.2]{Hani}. Thus, $J(R)=M$; and so $R$ is a quasi-local ring.
\end{proof}

Let $R$ be a ring and denote the set of all ideals of $R$ by $L(R)$. D. Zhao
\cite{Zhao} introduced the concept of expansions of ideals of the ring $R$. A
function $\delta:L(R)\rightarrow L(R)$ is called an ideal expansion if the
following conditions are satisfied for any ideals $I$ and $J$ of $R$ :

\begin{enumerate}
\item $I\subseteq\delta(I)$.

\item Whenever $I\subseteq J$, then $\delta(I)\subseteq$ $\delta(J)$.
\end{enumerate}

For example, $\delta_{1}:L(R)\rightarrow L(R)$ defined by $\delta_{1}%
(I)=\sqrt{I}$ is an ideal expansion of a ring $R$. For an ideal expansion
$\delta$ defined on a ring $R$, the class of $\delta$-$n$-ideals has been
defined and studied recently in \cite{Ece}. A proper ideal $I$ of $R$ is
called a $\delta$-$n$-ideal if whenever $a,b\in R$ and $ab\in I$, then $a\in
N(R)$ or $b\in\delta(I)$.

\begin{proposition}
\label{delta}Let $I$ be a proper ideal of $R$.
\end{proposition}

\begin{enumerate}
\item If $I$ is a $\delta_{1}$-$n$-ideal, then $I$ is a quasi $J$-ideal of
$R.$

\item If $I$ is a primary ideal of $R$ and $I\subseteq J(R)$, then $I$ is a
quasi $J$-ideal of $R$.
\end{enumerate}

\begin{proof}
(1) Suppose that $ab\in I$ and $a\notin J(R)$. Then $a\notin N(R)$ as
$N(R)\subseteq J(R).$ Since $I$ is a $\delta_{1}$-$n$-ideal, we have
$b\in\delta_{1}(I)=\sqrt{I}$. By Theorem \ref{eq}, we conclude that $I$ is a
quasi $J$-ideal of $R.$

(2) Suppose that $ab\in I$ and $a\notin J(R)$. If $b\notin\sqrt{I}$, then
$a\in I$ since $I$ is a primary ideal of $R$ which contradicts the assumption
that $I\subseteq J(R)$. Therefore, $b\in$ $\sqrt{I}$ and $I$ is a quasi
$J$-ideal by Theorem \ref{eq}.
\end{proof}

However, the converses of the implications in Proposition \ref{delta} are not
true in general as we can see in the following two examples.

\begin{example}
\label{edelta}Consider the quasi-local ring $%
\mathbb{Z}
_{\left\langle 2\right\rangle }=\left\{  \frac{a}{b}:a,b\in%
\mathbb{Z}
,\text{ }2\nmid b\right\}  $. Then $J(%
\mathbb{Z}
_{\left\langle 2\right\rangle })=$ $\left\langle 2\right\rangle _{\left\langle
2\right\rangle }=\left\{  \frac{a}{b}:a\in\left\langle 2\right\rangle ,2\nmid
b\right\}  $ is a quasi $J$-ideal of $%
\mathbb{Z}
_{\left\langle 2\right\rangle }$ by Theorem \ref{quasi}. On the other hand,
$\left\langle 2\right\rangle _{\left\langle 2\right\rangle }$ is not a
$\delta_{1}$-$n$-ideal. Indeed, if we take $\frac{2}{3},\frac{3}{5}\in%
\mathbb{Z}
_{\left\langle 2\right\rangle }$, then $\frac{2}{3}.\frac{3}{5}=\frac{6}%
{15}\in\left\langle 2\right\rangle _{\left\langle 2\right\rangle }$ but
$\frac{2}{3}\notin N(%
\mathbb{Z}
_{\left\langle 2\right\rangle })=0_{%
\mathbb{Z}
_{\left\langle 2\right\rangle }}$ and $\frac{3}{5}\notin\sqrt{\left\langle
2\right\rangle _{\left\langle 2\right\rangle }}=\left\langle 2\right\rangle
_{\left\langle 2\right\rangle }$.
\end{example}

\begin{example}
\label{eC}Consider the ring $C(%
\mathbb{R}
)$ of all real valued continuous functions and let $M=\left\{  f\in C(%
\mathbb{R}
):f(0)=0\right\}  $. Then $M$ is a maximal ideal of $C(%
\mathbb{R}
)$. Consider the quasi-local ring $R=\left(  C(%
\mathbb{R}
)\right)  _{M}$ and let $I=\left\langle x\sin x\right\rangle _{M}$. Then $I$
is a quasi $J$-ideal by Theorem \ref{quasi}. On the other hand $I$ is not
primary since for example $x\sin x\in I$ but $x^{n}\notin I$ and $\sin
^{n}x\notin I$ for all integers $n$.
\end{example}

Recall that a ring $R$ is said to be semiprimitive if $J(R)=0.$

\begin{proposition}
Let $R$ be a semiprimitive ring.
\end{proposition}

\begin{enumerate}
\item $R$ is an integral domain if and only if the only quasi $J$-ideal of $R$
is the zero ideal.

\item If $R$ is not an integral domain, then $R$ has no quasi $J$-ideals.
\end{enumerate}

\begin{proof}
(1) Suppose that $R$ is an integral domain. Then it is easy to show that $0$
is a quasi $J$-ideal of $R.$ If $I$ is a non-zero quasi $J$-ideal, then by
Proposition \ref{J(I)} we have $I\subseteq J(R)=0$ which is a contradiction.

(2) Suppose that $I$ is a quasi $J$-ideal of $R.$ Then $I\subseteq\sqrt
{I}\subseteq J(R)=0$. But since $R$ is not integral domain, then $0$ is not a
prime ideal of $R$ and so clearly it is not a quasi $J$-ideal.
\end{proof}

Let $R$ be a ring and $S$ be a non-empty subset of $R$. Then clearly $\left(
I:S\right)  =\left\{  r\in R:rS\subseteq I\right\}  $ is an ideal of $R$. Now,
while it is clear that $\sqrt{\left(  I:S\right)  }\subseteq\left(  \sqrt
{I}:S\right)  $, the reverse inclusion need not be true in general. For
example, consider $S=\left\{  2\right\}  \subseteq%
\mathbb{Z}
$ and the ideal $I=\left\langle 12\right\rangle $ of $%
\mathbb{Z}
$. Then $\sqrt{\left(  I:S\right)  }=\sqrt{\left\langle 6\right\rangle
}=\left\langle 6\right\rangle $ while $\left(  \sqrt{I}:S\right)
=\left\langle 3\right\rangle $.

\begin{lemma}
\label{1}If $I$ is a quasi $J$-ideal of a ring $R$ and $S\nsubseteq J(R)$ is a
subset of $R$, then $\sqrt{\left(  I:S\right)  }=\left(  \sqrt{I}:S\right)  $.
\end{lemma}

\begin{proof}
If $a\in\left(  \sqrt{I}:S\right)  $, then $sa\in\sqrt{I}$ for all $s\in S$.
Choose $s\notin J(R)$ such that $sa\in\sqrt{I}$. Then $a\in\sqrt{I}$ as $I$ is
a quasi $J$-ideal and so clearly, $a\in\sqrt{\left(  I:S\right)  }$. The other
inclusion is obvious.
\end{proof}

\begin{lemma}
\label{2}Let $S$ be a subset of a ring $R$ with $S\nsubseteq J(R)$ and $I$ be
a proper ideal of $R$. If $I$ is a quasi $J$-ideal, then $(I:S)$ is a quasi
$J$-ideal.
\end{lemma}

\begin{proof}
We first note that $(I:S)$ is proper in $R$ since otherwise if $1\in(I:S)$,
then $S\subseteq I\subseteq J(R)$, a contradiction. Suppose that $ab\in(I:S)$
and $a\notin J(R)$ for $a,b\in R$. Then $abS\subseteq I$ and $a\notin J(R)$
which imply that $bS\subseteq\sqrt{I}$ by Theorem 1. Thus, $b\in\left(
\sqrt{I}:S\right)  =\sqrt{\left(  I:S\right)  }$ by Lemma \ref{1} and we are done.
\end{proof}

A quasi $J$-ideal $I$ of a ring $R$ is called a maximal quasi $J$-ideal if
there is no quasi $J$-ideal which contains $I$ properly. In the following
proposition, we justify that any maximal quasi $J$-ideal is a $J$-ideal.

\begin{theorem}
\label{max}Let $I$ be a maximal quasi $J$-ideal of $R$. Then $I$ is a
$J$-ideal of $R$.
\end{theorem}

\begin{proof}
Suppose $I$ is a maximal quasi $J$-ideal of $R$. Let $a,b\in R$ such that
$ab\in I$ and $a\notin J(R)$. Then $\left(  I:a\right)  $ is a quasi $J$-ideal
of $R$ by Lemma \ref{2}. Since $I$ is a maximal quasi $J$-ideal and
$I\subseteq\left(  I:a\right)  $, then $b\in\left(  I:a\right)  =I$.
Therefore, $I$ is a $J$-ideal of $R$.
\end{proof}

If $J(R)$ is a quasi $J$-ideal of a ring $R$, then clearly it is the unique
maximal quasi $J$-ideal of $R$. In this case, $J(R)$ is a prime ideal of $R$
as can be seen in the following corollary.

\begin{corollary}
\label{J}Let $R$ be a ring. The following are equivalent:

\begin{enumerate}
\item $J(R)$ is a $J$-ideal of $R$.

\item $J(R)$ is a quasi $J$-ideal of $R$.

\item $J(R)$ is a prime ideal of $R$.
\end{enumerate}
\end{corollary}

Recall from \cite{F} that a proper ideal of a ring $R$ is called a quasi
primary ideal if its radical is prime. We prove in the following theorem that
under a certain condition on $R$, quasi primary ideals and quasi $J$-ideal are
the same.

\begin{theorem}
\label{zero}Let $R$ be a zero-dimensional ring and $I$ be an ideal of $R$ with
$I\subseteq J(R).$ Then the following are equivalent:
\end{theorem}

\begin{enumerate}
\item $I$ is a quasi $J$-ideal of $R.$

\item $I$ is a quasi primary ideal of $R.$

\item $I=P^{n}$ for some prime ideal $P$ of $R$ and some positive integer $n$.

\item $(R,\sqrt{I})$ is a quasi-local ring.
\end{enumerate}

\begin{proof}
(1)$\Rightarrow$(2) Suppose that $ab\in\sqrt{I}$ and $a\notin\sqrt{I}$. Then
there exists a positive number $n$ such that $a^{n}b^{n}\in I$. Since $R$ is
zero-dimensional, then every prime ideal is maximal and so $\sqrt{I}=J(I)$.
Since $I$ is a quasi $J$-ideal and clearly $a^{n}\notin J(I)$, we conclude
$b^{n}\in\sqrt{I}$ by Theorem \ref{J(I)}. Thus $b\in\sqrt{I}$ which shows that
$\sqrt{I}$ is prime as needed.

(2)$\Rightarrow$(3) Suppose that $I$ is a quasi primary ideal of $R$. Then
$\sqrt{I}$ is prime. Since $R$ is zero-dimensional, $\sqrt{I}$ is a maximal
ideal and clearly $I=P^{n}$ for some prime ideal $P$ of $R$ and some positive
number $n$.

(3)$\Rightarrow$(4) Suppose that $I=P^{n}$ for some prime ideal $P$ of $R$ and
some positive integer $n$. Then $\sqrt{I}=P$ is also a maximal ideal. Hence
our assumption $I\subseteq J(R)$ implies that $\sqrt{I}=P=J(R)$ and so
$(R,\sqrt{I})$ is a quasi-local ring.

(4)$\Rightarrow$(1) It follows directly by Theorem \ref{quasi}.
\end{proof}

Since every principal ideal ring is zero-dimensional, we have the following
corollary of Theorem \ref{zero}.

\begin{corollary}
Let $R$ be a principal ideal ring and $I$ be a proper ideal of $R$. Then $I$
is a quasi $J$-ideal of $R$ if and only if $I=p^{n}R$ for some prime element
$p$ of $R$ with $p\in J(R)$ and $n\geq1$.
\end{corollary}

Let $I$ be a proper ideal of $R$. Then $I$ is said to be superfluous if
whenever $K$ is an ideal of $R$ such that $I+K=R$, then $K=R.$

\begin{proposition}
If $I$ is a quasi $J$-ideal of a ring $R$, then $I$ is superfluous.
\end{proposition}

\begin{proof}
Suppose that $I+K=R$ for some ideal $K$ of $R$. Then $\sqrt{I}+\sqrt{K}%
=\sqrt{I+K}=R.$ From \cite[Proposition 2.9]{Hani}, we conclude that $\sqrt
{K}=R$ which means $K=R$ and we are done.
\end{proof}

\begin{proposition}
\begin{enumerate}
\item If $I_{1},I_{2},...,I_{k}$ are quasi $J$-ideals of a ring $R$, then $%
{\displaystyle\bigcap\limits_{i=1}^{k}}
I_{i}$ is a quasi $J$-ideal of $R$.

\item Let $I_{1},I_{2},...,I_{k}$ be quasi primary ideals of a ring $R$ in
which their radicals are not comparable. If $%
{\displaystyle\bigcap\limits_{i=1}^{k}}
I_{i}$ is a quasi $J$-ideal of $R$, then $I_{i}$ is a quasi $J$-ideal of $R$
for $i=1,2,...,k$.
\end{enumerate}
\end{proposition}

\begin{proof}
(1) Since $\sqrt{%
{\displaystyle\bigcap\limits_{i=1}^{k}}
I_{i}}=%
{\displaystyle\bigcap\limits_{i=1}^{k}}
\sqrt{I_{i}}$, the claim is clear by \cite[Proposition 2.25]{Hani}.

(2) Without loss of generality, we show that $I_{1}$ is a quasi $J$-ideal.
Suppose that $ab\in I_{1}$ and $a\notin J(R).$ By assumption, we can choose an
element $c\in\left(
{\displaystyle\prod\limits_{i=2}^{k}}
I_{i}\right)  \backslash\sqrt{I_{1}}$ and then we have $abc\in%
{\displaystyle\bigcap\limits_{i=1}^{k}}
I_{i}$. It follows that $bc\in\sqrt{%
{\displaystyle\bigcap\limits_{i=1}^{k}}
I_{i}}=%
{\displaystyle\bigcap\limits_{i=1}^{k}}
\sqrt{I_{i}}\subseteq\sqrt{I_{1}}$ as $%
{\displaystyle\bigcap\limits_{i=1}^{k}}
I_{i}$ is a quasi $J$-ideal. Since $I_{1}$ is quasi primary, $\sqrt{I_{1}}$ is
prime which implies that $b\in\sqrt{I_{1}}.$ Thus $I_{1}$ is a quasi $J$-ideal
of $R.$
\end{proof}

\begin{proposition}
\begin{enumerate}
\item Let $I_{1},I_{2},...,I_{k}$ be quasi $J$-ideals of a ring $R$. Then
$\prod\limits_{i=1}^{k}I_{i}$ is a quasi $J$-ideal of $R$.

\item Let $I_{1},I_{2},...,I_{k}$ be quasi primary ideals of $R$ in which
their radicals are not comparable. If $\prod\limits_{i=1}^{k}I_{i}$ is a quasi
$J$-ideal of $R$, then $I_{i}$ is a quasi $J$-ideal of $R$ for $i=1,2,...,k.$
\end{enumerate}
\end{proposition}

\begin{proof}
(1) Let $a,b\in R$ such that $ab\in\prod\limits_{i=1}^{k}I_{i}$ and $a\notin
J(R)$. Then clearly for all $i=1,2,...,k$, $b\in\sqrt{I_{i}}$ since $I_{i}$ is
a quasi $J$-ideal of $R$. Now, for all $i$, there is an integer $n_{i}$ such
that $b^{n_{i}}\in I_{i}$. Thus, $b^{n_{1}+n_{2}+\cdots+n_{k}}\in
\prod\limits_{i=1}^{k}I_{i}$ and so $b\in\sqrt{\prod\limits_{i=1}^{k}I_{i}}$.
Therefore, $\prod\limits_{i=1}^{k}I_{i}$ is a quasi $J$-ideal.

(2) Similar to the proof of Proposition 5 (2).
\end{proof}

However, the $J$-ideal property can not pass to the product of ideals as can
be seen in the following example.

\begin{example}
Consider the ring $%
\mathbb{Z}
(+)%
\mathbb{Z}
_{2}$. Then $0(+)%
\mathbb{Z}
_{2}$ is a $J$-ideal since $0$ is a $J$-ideal of $%
\mathbb{Z}
$. But $(0(+)%
\mathbb{Z}
_{2})(0(+)%
\mathbb{Z}
_{2})=0(+)\overline{0}$ is not a $J$-ideal of $%
\mathbb{Z}
(+)%
\mathbb{Z}
_{2}$ since for example, $(2,\overline{0})(0,\overline{1})=(0,\overline{0})$
and $(2,\overline{0})\notin J(%
\mathbb{Z}
)(+)%
\mathbb{Z}
_{2}=J(%
\mathbb{Z}
(+)%
\mathbb{Z}
_{2})$ but $(0,\overline{1})\neq(0,\overline{0}).$
\end{example}

\begin{proposition}
\label{f}Let $R_{1}$ and $R_{2}$ be two rings and $f:R_{1}\rightarrow R_{2}$
be an epimorphism. Then the following statements hold:
\end{proposition}

\begin{enumerate}
\item If $I_{1}$ is a quasi $J$-ideal of $R_{1}$ with $K\operatorname{erf}%
\subseteq I_{1}$, then $f(I_{1})$ is a quasi $J$-ideal of $R_{2}.$

\item If $I_{2}$ is a quasi $J$-ideal of $R_{2}$ and $K\operatorname{erf}%
\subseteq J(R)$, then $f^{-1}(I_{2})$ is a quasi $J$-ideal of $R_{1}.$
\end{enumerate}

\begin{proof}
(1) Suppose that $I_{1}$ is a quasi $J$-ideal of $R_{1}.$ Since $\sqrt{I_{1}}$
is a $J$-ideal of $R_{1}$ and $K\operatorname{erf}\subseteq I_{1}%
\subseteq\sqrt{I_{1}}$, then $f(\sqrt{I_{1}})$ is a $J$-ideal of $R_{2}$ by
\cite[Proposition 2.23]{Hani}. Now, if $a,b\in R_{2}$ such that $ab\in
\sqrt{f(I_{1})}$ and $a\notin J(R_{2})$, then $a^{n}b^{n}\in f(I_{1})\subseteq
f(\sqrt{I_{1}})$ for some integer $n$. Since $a^{n}\notin J(R_{2})$, then
$b^{n}\in f(\sqrt{I_{1}})\subseteq\sqrt{f(I_{1})}$. Therefore, $b\in
\sqrt{f(I_{1})}$ and $\sqrt{f(I_{1})}$ is a $J$-ideal of $R_{2}$. So,
$f(I_{1})$ is a quasi $J$-ideal of $R_{2}.$

(2) Suppose that $I_{2}$ is a quasi $J$-ideal of $R_{2}.$ Since $\sqrt{I_{2}}$
is a $J$-ideal of $R_{2}$ and $K\operatorname{erf}\subseteq J(R)$, then
$f^{-1}(\sqrt{I_{2}})$ is a $J$-ideal of $R_{1}$ by \cite[Proposition
2.23]{Hani}. Now, let $x,y\in R_{1}$ such that $xy\in\sqrt{f^{-1}(I_{2})}$ and
$x\notin J(R_{1})$. Then $x^{m}y^{m}\in f^{-1}(I_{2})\subseteq f^{-1}%
(\sqrt{I_{2}})$ for some integer $m$. But $x^{m}\notin J(R_{1})$ implies that
$y^{m}\in f^{-1}(\sqrt{I_{2}})\subseteq\sqrt{f^{-1}(I_{2})}$. It follows that
$y\in\sqrt{f^{-1}(I_{2})};$ and so $\sqrt{f^{-1}(I_{2})}$ is a $J$-ideal of
$R_{1}$.
\end{proof}

\begin{corollary}
Let $I$ and $K$ be proper ideals of $R$ with $K\subseteq I$. If $I$ is a quasi
$J$-ideal of $R$, then $I/K$ is a quasi $J$-ideal of $R/K$.
\end{corollary}

\begin{proof}
Consider the natural epimorphism $\pi:R\rightarrow R/K$ with $Ker(\pi
)=K\subseteq I.$ By Proposition \ref{f}, $\pi(I)=I/K$ is a quasi $J$-ideal of
$R/K$.
\end{proof}

Let $I$ be a proper ideal of $R.$ In the following, the notation $Z_{I}%
(R)$\ denotes the set of $\{r\in R|rs\in I$ for some $s\in R\backslash I\}.$

\begin{proposition}
\label{S}Let $S$ be a multiplicatively closed subset of a ring $R$ such that
$J(S^{-1}R)=S^{-1}J(R)$. Then the following hold:
\end{proposition}

\begin{enumerate}
\item If $I$ is a quasi $J$-ideal of $R$ such that $I\cap S=\emptyset$, then
$S^{-1}I$ is a quasi $J$-ideal of $S^{-1}R.$

\item If $S^{-1}I$ is a quasi $J$-ideal of $S^{-1}R$ and $S\cap Z_{I}(R)=S\cap
Z_{J(R)}(R)=\emptyset$, then $I$ is a quasi $J$-ideal of $R.$
\end{enumerate}

\begin{proof}
(1)\ Suppose that $I$ is a quasi $J$-ideal of $R.$ Since $\sqrt{I}$ is a
$J$-ideal of $R$, then by \cite[Proposition 2.26]{Hani}, we conclude that
$\sqrt{S^{-1}I}=S^{-1}\sqrt{I}$ is a $J$-ideal of $R$ and we are done.

(2) Let $a,b\in R$ and $ab\in I$. Hence $\frac{a}{1}\frac{b}{1}\in S^{-1}I$.
Since $S^{-1}I$ is a quasi $J$-ideal of $S^{-1}R$, we have either $\frac{a}%
{1}\in J(S^{-1}(R))=S^{-1}J(R)$ or $\frac{b}{1}\in\sqrt{S^{-1}I}=S^{-1}%
\sqrt{I}$ by Theorem \ref{eq}$.$ If $\frac{b}{1}\in S^{-1}\sqrt{I}$, then
there exist $u\in S$ and a positive integer $n$ such that $u^{n}b^{n}\in I$.
Since $S\cap Z_{I}(R)=\emptyset,$ we conclude that $b^{n}\in I$ and so
$b\in\sqrt{I}.$ If $\frac{a}{1}\in S^{-1}J(R)$, then there exist $v\in S$ and
a positive integer $m$ such that $v^{m}a^{m}\in J(R)$. Since $S\cap
Z_{J(R)}(R)=\emptyset,$ we conclude that $a^{m}\in J(R)$ and so $a\in J(R).$
Therefore, $I$ is a quasi $J$-ideal of $R$ by Theorem \ref{eq}.
\end{proof}

Next, we justify that decomposable rings have no $J$-ideals.

\begin{remark}
\label{r}Let $R_{1}$ and $R_{2}$ be two rings and $R=R_{1}\times R_{2}$. Then
there are no quasi $J$-ideal in $R$. Indeed, for every proper ideal
$I_{1}\times I_{2}$ of $R$ we have $(1,0)(0,1)\in I_{1}\times I_{2}$ but
neither $(1,0)\in J(R)$ nor $(0,1)\in\sqrt{I_{1}\times I_{2}}=$ $\sqrt{I_{1}%
}\times\sqrt{I_{2}}$.
\end{remark}

\begin{lemma}
\label{pow}Let $I$ be an ideal of a Noetherian ring $R$. Then $\sqrt
{I[\left\vert x\right\vert ]}=\sqrt{I}[\left\vert x\right\vert ]$.
\end{lemma}

\begin{proof}
See \cite{1}.
\end{proof}

\begin{proposition}
\label{pp}Let $I$ be a proper ideal of a Noetherian ring $R$. Then
$I[\left\vert x\right\vert ]$ is a quasi $J$-ideal of $R[|x|]$ if and only if
$I$ is a quasi $J$-ideal of $R$.
\end{proposition}

\begin{proof}
Follows by \cite[Proposition 2.18]{Hani} and Lemma \ref{pow}.
\end{proof}

\section{Quasi presimplifiable rings}

Recall that a ring $R$ is called presimplifiable if whenever $a,b\in R$ with
$a=ab$, then $a=0$ or $b\in U(R)$. This class of rings has been introduced by
Bouvier in \cite{Bouvier}. Then many of its properties are studied in
\cite{Anderson} and \cite{And2}. Among many other characterizations, it is
well known that $R$ is presimplifiable if and only if $Z(R)\subseteq J(R)$. As
a generalization of presimplifiable property, we introduce the following class
of rings.

\begin{definition}
A ring $R$ is called quasi presimplifiable if whenever $a,b\in R$ with $a=ab$,
then $a\in N(R)$ or $b\in U(R)$.
\end{definition}

It is clear that any presimplifiable ring $R$ is quasi presimplifiable and
that they coincide if $R$ is reduced. The following example shows that the
converse is not true in general.

\begin{example}
\label{exp}Let $R=%
\mathbb{Z}
(+)%
\mathbb{Z}
_{2}$ and let $(a,m_{1}),(b,m_{2})\in R$ such that $(a,m_{1})(b,m_{2}%
)=(a,m_{1})$ and $(a,m_{1})\notin N(R)=N(%
\mathbb{Z}
)(+)%
\mathbb{Z}
_{2}$. Then $ab=a$ with $a\notin N(R)$ and so we must have $b=1\in U(%
\mathbb{Z}
)$. It follows that $(b,m_{2})\in U(%
\mathbb{Z}
)(+)%
\mathbb{Z}
_{2}=U(%
\mathbb{Z}
(+)%
\mathbb{Z}
_{2})=U(R)$ and $R$ is quasi presimplifiable. On the other hand, $R$ is not
presimplifiable. For example $(0,\overline{1}),(3,\overline{1})\in R$ and
$(0,\overline{1})(3,\overline{1})=(0,\overline{1})$ but $(0,\overline
{1}),(3,\overline{1})\neq(0,\overline{0})$ and $(0,\overline{1}),(3,\overline
{1})\notin U(R)$.
\end{example}

A non-zero element $a$ in a ring $R$ is called quasi-regular if $Ann_{R}%
(a)\subseteq N(R)$. We denote the set of all elements of $R$ that are not
quasi-regular by $NZ(R)$. As a characterization of quasi presimplifiable
rings, we have the following.

\begin{proposition}
\label{q1}A ring $R$ is quasi presimplifiable if and only if $NZ(R)\subseteq
J(R)$.
\end{proposition}

\begin{proof}
Suppose $R$ is quasi presimplifiable, $a\in NZ(R)$ and $r\in R$. Then $ra\in
NZ(R)$ and so there exists $b\notin N(R)$ such that $rab=0$. Hence,
$(1-ra)b=b$ and so by assumption, $1-ra\in U(R)$. It follows that $a\in J(R)$
and so $NZ(R)\subseteq J(R)$. Conversely, suppose $NZ(R)\subseteq J(R)$ and
let $a,b\in R$ with $a=ab$. Then $a(1-b)=0$. If $a\in N(R)$, then we are done,
otherwise, $1-b\in NZ(R)\subseteq J(R)$. Therefore, $b\in U(R)$ as required.
\end{proof}

The main result of this section is to clarify the relationship between quasi
$J$-ideals (resp. $J$-ideals) and quasi presimplifiable (resp.
presimplifiable) rings.

\begin{theorem}
\label{p/}Let $I$ be a proper ideal of a ring $R$. Then
\end{theorem}

\begin{enumerate}
\item $I$ is a $J$-ideal of $R$ if and only if $I\subseteq J(R)$ and $R/I$ is presimplifiable.

\item $I$ is a quasi $J$-ideal of $R$ if and only if $I\subseteq J(R)$ and
$R/I$ is quasi presimplifiable.
\end{enumerate}

\begin{proof}
\begin{enumerate}
\item Suppose $I$ is a $J$-ideal of $R$. Then $I\subseteq J(R)$ by
\cite[Proposition 2.2]{Hani}. Now, let $a+I\in Z(R/I)$. Then there exists
$I\neq b+I\in R/I$ such that $(a+I)(b+I)=I$. Now, $ab\in I$ and $b\notin I$
imply that $a\in J(R)$ as $I$ is a $J$-ideal of $R$. Thus, $a+I\in
J(R)/I=J(R/I)$ and so $R/I$ is presimplifiable. Conversely, suppose $R/I$ is
presimplifiable and let $a,b\in R$ such that $ab\in I$ and $a\notin J(R)$.
Then $a+I\notin J(R)/I=J(R/I)$ and by assumption, $a+I\notin Z(R/I)$. As
$(a+I)(b+I)=I$, we conclude that $b+I=I$ and so $b\in I$ as needed.

\item Suppose $I$ is a quasi $J$-ideal of $R$ and note that $I\subseteq J(R)$
by Proposition \ref{J(I)}. Let $a+I\in NZ(R/I)$ and choose $b+I\notin N(R/I)$
such that $(a+I)(b+I)=I$. Then $ab\in I$ and $b\notin\sqrt{I}$ which imply
that $a\in J(R)$ as $I$ is a quasi $J$-ideal of $R$. Hence, $a+I\in
J(R)/I=J(R/I)$ and $R/I$ is quasi presimplifiable by Proposition \ref{q1}.
Conversely, suppose $R/I$ is quasi presimplifiable and let $a,b\in R$ such
that $ab\in I$ and $a\notin J(R)$. Then $a+I\notin J(R)/I=J(R/I)$ and so
$a+I\notin NZ(R/I)$. As $(a+I)(b+I)=I$, we must have $b+I\in N(R/I)$ and so
$b\in\sqrt{I}$. Therefore, $I$ is a quasi $J$-ideal.
\end{enumerate}
\end{proof}

In view of Theorem \ref{p/}, we deduce immediately the following
characterization of presimplifiable (resp. quasi presimplifiable) rings.

\begin{corollary}
A ring $R$ is presimplifiable (resp. quasi presimplifiable) if and only if $0$
is a $J$-ideal (resp. quasi $J$-ideal) of $R$.
\end{corollary}

Recall that a ring $R$ is said to be von Neumann regular if for every $a\in
R$, there exists an element $x\in R$ such that $a=a^{2}x$.

\begin{lemma}
\label{lvon}If $R$ is a quasi presimplifiable von Neumann regular ring, then
$R$ is a field.
\end{lemma}

\begin{proof}
Let $a$ be a non-zero element of $R$. Since $R$ is von Neumann regular,
$a=a^{2}x$ for some element $x$ of $R$. Observe that $a\notin N(R)$ as every
von Neumann regular ring is reduced. Since $a=a(ax)$ and $R$ is quasi
presimplifiable, we conclude that $ax\in U(R)$ and so $a\in U(R)$. Thus, $R$
is a field.
\end{proof}

We call an ideal $I$ of a ring $R$ regular if $R/I$ is a von Neumann regular ring.

\begin{proposition}
Any regular quasi $J$-ideal in a ring $R$ is maximal.
\end{proposition}

\begin{proof}
Suppose $I$ is a regular quasi $J$-ideal of $R$. Then $R/I$ is a von Neumann
regular ring. Moreover, as $I\subseteq J(R)$, then $R/I$ is quasi
presimplifiable by Theorem \ref{p/}. It follows by Lemma \ref{lvon} that $R/I$
is a field and so $I$ is maximal in $R$.
\end{proof}

For a ring $R$, we recall that $f(x)=%
{\displaystyle\sum\limits_{i=0}^{n}}
a_{i}x^{i}\in R[x]$ is a unit if and only if $a_{0}\in U(R)$ and $a_{1}%
,a_{2},...,a_{n}\in N(R)$. In \cite{Anderson}, it has been proved that $R[x]$
is presimplifiable if and only if $R$ is presimplifiable and $0$ is a primary
ideal of $R$.

\begin{proposition}
\label{ps}Let $R$ be a ring. Then $R[x]$ is quasi presimplifiable if and only
if $R$ is quasi presimplifiable and $0$ is a $\delta_{1}$-$n$-ideal of $R$.
\end{proposition}

\begin{proof}
Suppose that $R[x]$ is presimplifiable and let $a,b\in R\subseteq R[x]$ such
that $a=ab$ and $a\notin N(R)$. Then $a\notin N(R[x])$ and so by our
assumption $b\in U(R[x])$. It follows that $b\in U(R)$ and so $R$ is quasi
presimplifiable. Now, let $a,b\in R$ such that $ab=0$ and $a\notin N(R)$. Then
we have $a=a(1-bx)$ and so $1-bx\in U(R[x])$. Hence $b\in N(R)$ and $0$ is a
$\delta_{1}$-$n$-ideal. For the converse, let $f(x)=%
{\displaystyle\sum\limits_{i=0}^{n}}
a_{i}x^{i}$ , $g(x)=%
{\displaystyle\sum\limits_{j=0}^{m}}
b_{j}x^{j}$ $\in R[x]$ such that $f(x)=f(x)g(x)$ and $f(x)\notin N(R[x])$.
Then $a_{i}\notin N(R)$ for some $i$. Now, $a_{i}=g(x)a_{i}$ implies that
$a_{i}=b_{0}a_{i}$ and so $b_{0}\in U(R)$ since $R$ is quasi presimplifiable.
Moreover, for all $j\neq0$, we have $a_{i}b_{j}=0$. So, $b_{j}\in N(R)$ for
all $j\neq0$ as $0$ is a $\delta_{1}$-$n$-ideal of $R$. Therefore, $g(x)\in
U(R[x])$ and $R[x]$ is quasi presimplifiable.
\end{proof}

Recall that a ring $R$ is called a Hilbert ring if every prime ideal of $R$ is
an intersection of maximal ideals. Moreover, it is well known that $R$ is a
Hilbert ring if and only if $M\cap R$ is a maximal ideal of $R$ whenever $M $
is a maximal ideal of $R[x]$, see \cite{Hilbert}. In this case, we have
$J(R)[x]\subseteq J(R[x])$. Indeed, if $M$ is a maximal ideal of $R[x]$, then
$M\cap R$ is a maximal ideal of $R$. Hence, $J(R)[x]\subseteq(M\cap
R)[x]\subseteq M$.

In the following proposition, we determine conditions under which the
extension $I[x]$ in $R[x]$ is a quasi $J$-ideal.

\begin{proposition}
\label{pol}Let $I$ be an ideal of a ring $R$.
\end{proposition}

\begin{enumerate}
\item If $I[x]$ is a quasi $J$-ideal of $R[x]$, then $I$ is a quasi $J$-ideal
of $R$.

\item If $R$ is Hilbert, $I\subseteq J(R)$ and $I$ is a $\delta_{1}$-$n$-ideal
of $R$, then $I[x]$ is a quasi $J$-ideal of $R[x]$.
\end{enumerate}

\begin{proof}
\begin{enumerate}
\item Suppose $I[x]$ is a quasi $J$-ideal of $R[x]$ and let $a,b\in R\subseteq
R[x]$ such that $ab\in I\subseteq I[x]$ and $a\notin J(R)$. Then clearly
$a\notin J(R[x])$ and so $b\in\sqrt{I[x]}$. It follows clearly that $b\in
\sqrt{I}$ and so $I$ is a quasi $J$-ideal of $R$.

\item Suppose $I$ is a $\delta_{1}$-$n$-ideal of $R$. Then $I$ is a quasi
$J$-ideal by Proposition \ref{delta}. Thus, $R/I$ is quasi presimplifiable by
Theorem \ref{p/} (2). By Proposition \ref{ps}, we conclude that
$R[x]/I[x]\cong(R/I)[x]$ is also quasi presimplifiable. Moreover, since $R$ is
Hilbert, then $I[x]\subseteq J(R)[x]\subseteq J(R[x])$. Therefore, $I[x]$ is a
quasi $J$-ideal of $R[x]$ again by Theorem \ref{p/} (2).
\end{enumerate}
\end{proof}

Recall that $(\Lambda,\leq)$ is called a directed quasi-ordered set if $\leq$
is a reflexive and transitive relation on $\Lambda$ and for $\alpha,\beta
\in\Lambda$, there exists $\gamma\in\Lambda$ with $\alpha\leq\gamma$ and
$\beta\leq\gamma$. A system of rings over $(\Lambda,\leq)$ is a collection
$\{R_{\alpha}:\alpha\in\Lambda\}$ of rings, together with ring homomorphisms
$\varphi_{\alpha,\beta}:R_{\alpha}\rightarrow R_{\beta}$ for all $\alpha
,\beta\in\Lambda$ with $\alpha\leq\beta$ such that $\varphi_{\beta,\gamma
}\circ\varphi_{\alpha,\beta}=\varphi_{\alpha,\gamma}$ whenever $\alpha
\leq\beta\leq\gamma$ and such that $\varphi_{\alpha,\alpha}=Id_{R_{\alpha}}$
for all $\alpha$. A direct limit of $\{R_{\alpha}:\alpha\in\Lambda\}$ is a
ring $R$ together with ring homomorphisms $\varphi_{\alpha}:R_{\alpha
}\rightarrow R$ such that $\varphi_{\beta}\circ\varphi_{\alpha,\beta}%
=\varphi_{\alpha}$ for all $\alpha,\beta\in\Lambda$ with $\alpha\leq\beta$ and
such that following property is satisfied: For any ring $S$ and collection
$\{f_{\alpha}:\alpha\in\Lambda\}$ of ring maps $f:R_{\alpha}\rightarrow S$
such that $f_{\beta}\circ\varphi_{\alpha,\beta}=f_{\alpha}$ for all
$\alpha,\beta\in\Lambda$ with $\alpha\leq\beta$, there is a unique ring
homomorphism $f:R\rightarrow S$ with $f\circ\varphi_{\alpha}=f_{\alpha}$ for
all $\alpha\in\Lambda$. This direct limit is usually denoted by
$R=\underrightarrow{\lim}R_{\alpha}$.

\begin{lemma}
\cite{Allen} Let $\{R_{\alpha}:\alpha\in\Lambda\}$ be a system of rings and
let $R=\underrightarrow{\lim}R_{\alpha}$. If $\{I_{\alpha}:\alpha\in\Lambda\}$
is a family of ideals over $\{R_{\alpha}:\alpha\in\Lambda\}$, then
$I=\sum\limits_{\alpha\in\Lambda}\varphi_{\alpha}(I_{\alpha})=\bigcup
\limits_{\alpha\in\Lambda}\varphi_{\alpha}(I_{\alpha})$ is an ideal of $R$.
Moreover, $R/I=\underrightarrow{\lim}R_{\alpha}/I_{\alpha}$.
\end{lemma}

In \cite{And2}, it is proved that if $\{R_{\alpha}:\alpha\in\Lambda\}$ is a
system of presimplifiable rings, then so is $R=\underrightarrow{\lim}%
R_{\alpha}$. In the following proposition, we generalize this result to quasi
presimplifiable case.

\begin{proposition}
\label{dlim}Let $(\Lambda,\leq)$ be a directed quasi-ordered set and let
$\{R_{\alpha}:\alpha\in\Lambda\}$ be a direct system of rings. If each
$R_{\alpha}$ is quasi presimplifiable, then the direct limit
$R=\underrightarrow{\lim}R_{\alpha}$ is quasi presimplifiable.
\end{proposition}

\begin{proof}
Let $x,y\in R$ with $x=xy$ and $x\notin N(R)$. For $\alpha\in\Lambda$, let
$\varphi_{\alpha}:R_{\alpha}\rightarrow R$ be the natural map. Then there
exist $\alpha_{0}\in\Lambda$ and $x_{\alpha_{0}},y_{\alpha_{0}}\in
R_{\alpha_{0}}$ such that $\varphi_{\alpha_{0}}(x_{\alpha_{0}})=x$ ,
$\varphi_{\alpha_{0}}(y_{\alpha_{0}})=y$ and $x_{\alpha_{0}}y_{\alpha_{0}%
}=x_{\alpha_{0}}$. Since $x\notin N(R)$, then $x_{\alpha_{0}}\notin
N(R_{\alpha_{0}})$, see \cite{Attiyah}, and so $y_{\alpha_{0}}\in
U(R_{\alpha_{0}})$ as $R_{\alpha_{0}}$ is quasi presimplifiable. Therefore,
$y=\varphi_{\alpha_{0}}(y_{\alpha_{0}})\in U(R)$ and so $R$ is quasi presimplifiable.
\end{proof}

\begin{theorem}
\label{dun}Let $(\Lambda,\leq)$ be a directed quasi-ordered set and let
$\{R_{\alpha}:\alpha\in\Lambda\}$ be a direct system of rings. If
$\{I_{\alpha}:\alpha\in\Lambda\}$ is a family of $J$-deals (resp. quasi
$J$-ideals) over $\{R_{\alpha}:\alpha\in\Lambda\}$, then $I=\bigcup
\limits_{\alpha\in\Lambda}\varphi_{\alpha}(I_{\alpha})$ is a $J$-ideal (resp.
quasi $J$-ideal) of $R=\underrightarrow{\lim}R_{\alpha}$.
\end{theorem}

\begin{proof}
For all $\alpha\in\Lambda$, we have $I_{\alpha}\subseteq J(R_{\alpha})$.
Hence, $I=\bigcup\limits_{\alpha\in\Lambda}\varphi_{\alpha}(I_{\alpha
})\subseteq\bigcup\limits_{\alpha\in\Lambda}\varphi_{\alpha}(J(R_{\alpha
}))\subseteq J(\underrightarrow{\lim}R_{\alpha})=J(R)$. Indeed, let
$x\in\bigcup\limits_{\alpha\in\Lambda}\varphi_{\alpha}(J(R_{\alpha}))$ and
$r\in R$. Then there exist $\alpha_{0}\in\Lambda$ and $x_{\alpha_{0}%
},r_{\alpha_{0}}\in R_{\alpha_{0}}$ such that $\varphi_{\alpha_{0}}%
(x_{\alpha_{0}})=x$ and $\varphi_{\alpha_{0}}(r_{\alpha_{0}})=r$. Now,
$1-rx=\varphi_{\alpha_{0}}(1_{R_{\alpha_{0}}}-r_{\alpha_{0}}x_{\alpha_{0}}%
)\in\varphi_{\alpha_{0}}(U(R_{\alpha_{0}}))\subseteq U(R)$ and so $x\in J(R)$.
Since for all $\alpha\in\Lambda$, $I_{\alpha}$ is a $J$-ideal (quasi
$J$-ideal), then $R_{\alpha}/I_{\alpha}$ is a presimplifiable (quasi
presimplifiable) ring by Theorem \ref{p/}. This implies that
$R/I=\underrightarrow{\lim}R_{\alpha}/I_{\alpha}$ is presimplifiable (quasi
presimplifiable) by Proposition \ref{dlim}. It follows again by Theorem
\ref{p/} that $I$ is a $J$-ideal (quasi $J$-ideal) of $R$.
\end{proof}

Finally, for a ring $R$, an ideal $I$ of $R$ and an $R$-module $M$, we
determine when is the ideal $I(+)M$ quasi $J$-ideal in $R(+)M$.

\begin{proposition}
\label{pide}Let $I$ be an ideal of a ring $R$ and let $M$ be an $R$-module.
Then $I(+)M$ is a quasi $J$-ideal of $R(+)M$ if and only if $I$ is a quasi
$J$-ideal of $R$.
\end{proposition}

\begin{proof}
We have $I(+)M\subseteq J(R(+)M)$ if and only if $I\subseteq J(R)$ and
$R/I\cong R(+)M/I(+)M$. Therefore, the result follows directly by Theorem
\ref{p/}.
\end{proof}

\end{document}